\documentclass[12pt]{amsart}
\usepackage{amssymb,latexsym,amsmath,enumerate,nicefrac}
\usepackage{amsaddr}

\usepackage{color}
\definecolor{myurlcolor}{rgb}{0.6,0,0}
\definecolor{mycitecolor}{rgb}{0,0,0.8}
\definecolor{myrefcolor}{rgb}{0,0,0.8}
\usepackage[bookmarks=false,pagebackref=true]{hyperref}
\hypersetup{colorlinks,
linkcolor=myrefcolor,
citecolor=mycitecolor,
urlcolor=myurlcolor}

\newtheorem{theorem}{Theorem}[section]
\newtheorem{lemma}[theorem]{Lemma}

\newtheorem{proposition}[theorem]{Proposition}

\theoremstyle{definition}
\newtheorem{definition}[theorem]{Definition}
\newtheorem{example}[theorem]{Example}
\theoremstyle{remark}

\usepackage[all]{xy}
\usepackage{comment}
 \usepackage{tikz}
\usetikzlibrary{arrows}
 \usetikzlibrary{matrix,positioning}

\DeclareMathOperator{\Hom}{Hom}

\allowdisplaybreaks
\includecomment{pf}

\excludecomment{expo}

\excludecomment{cut}

\makeatletter
\newcommand*\bigcdot{\mathpalette\bigcdot@{.6}}
\newcommand*\bigcdot@[2]{\mathbin{\vcenter{\hbox{\scalebox{#2}{$\m@th#1\bullet$}}}}}
\makeatother

\newcounter{BWYtable}

\newcounter{BWYDiagram}

\newcounter{BWYFigure}

\renewcommand{\1}{1}
\renewcommand{\2}{2}
\newcommand{\3}{3}
\newcommand{\4}{4}

\newcommand{\C}{\mathscr{C}}

\usepackage{tikz}
\usetikzlibrary{calc}
\usepackage{graphics}
\usepackage{dsfont}
\usepackage{amsmath}
\usepackage{amsfonts,amssymb}
\usepackage{mathrsfs}
\usepackage{mathtools}
\usepackage{comment}

\renewcommand{\(}{\left(} 
\renewcommand{\)}{\right)}

\newcommand{\catname}[1]{{\mathsf{#1}}} 

\newcommand{\Set}{\catname{Set}}

\newcommand{\Dif}{{\catname{Diff}}}
\newcommand{\SurjSub}{{\catname{SurjSub}}}
\newcommand{\Riem}{{\catname{RiemSurj}}}
\newcommand{\SympSurj}{{\catname{SympSurj}}}

\newcommand{\Top}{{\catname{Top}}}

\newcommand{\F}{\mathcal F}

\newcommand{\SA}{S_A}
\newcommand{\SL}{S_L}
\newcommand{\SR}{S_R}
\newcommand{\sL}{s_L}
\newcommand{\sR}{s_R}

\newcommand{\tsL}{\tilde{s}_L}
\newcommand{\tsR}{\tilde{s}_R}

\newcommand{\TA}{T_A}
\newcommand{\TL}{T_L}
\newcommand{\TR}{T_R}
\newcommand{\tL}{t_L}
\newcommand{\tR}{t_R}

\newcommand{\CA}{C_A}
\newcommand{\CL}{C_L}
\newcommand{\CR}{C_R}
\newcommand{\cL}{c_L}
\newcommand{\cR}{c_R}

\newcommand{\tcL}{\tilde{c}_L}
\newcommand{\tcR}{\tilde{c}_R}

\newcommand{\QA}{Q_A}
\newcommand{\QL}{Q_L}
\newcommand{\QR}{Q_R}
\newcommand{\qL}{q_L}
\newcommand{\qR}{q_R}

\newcommand{\PA}{P_A}
\newcommand{\PL}{P_L}
\newcommand{\PR}{P_R}
\newcommand{\pL}{p_L}
\newcommand{\pR}{p_R}

\renewcommand{\d}{{\rm d}}

\pgfkeys{/tikz/.cd, arrowhead/.default=-1pt, arrowhead/.code={}}

\begin{document}

\title{Constructing Span Categories From Categories Without Pullbacks}


\author{David Weisbart$^1$ \and Adam M. Yassine$^2$}
\address{
\begin{tabular}[h]{cc}
 $^{1}$Department of Mathematics &  $^{2}$Department of Mathematics\\
  University of California, Riverside &  Bowdoin College
  \end{tabular}
  }
\email{$^1$weisbart@math.ucr.edu} \email{$^2$a.yassine@bowdoin.edu}

\maketitle

\begin{flushright}
In memory of Professor V.S.Varadarajan.
\end{flushright}

\pagestyle{plain}

\begin{abstract}
Span categories provide an abstract framework for formalizing mathematical models of certain systems.  The mathematical descriptions of some systems, such as classical mechanical systems, require categories that do not have pullbacks, and this limits the utility of span categories as a formal framework.  Given categories $\C$ and $\C^\prime$ and a functor $\F$ from $\C$ to $\C^\prime$, we introduce the notion of an $\F$-pullback of a cospan in $\C$, as well as the notion of span tightness of $\F$. If $\mathcal F$ is span tight, then we can form a generalized span category ${\rm Span}(\C,\F)$ and circumvent the technical difficulty of $\C$ failing to have pullbacks.  Composition in ${\rm Span}(\C,\F)$ uses $\F$-pullbacks rather than pullbacks and in this way differs from the category ${\rm Span}(\C)$, but reduces to it when both $\C$ has pullbacks and $\F$ is the identity functor.  
\end{abstract}


\tableofcontents

\section{Introduction}\label{sec:intro}

The last few decades have seen significant application of category theory to modeling systems.  The basic idea is to model a system by a morphism in an appropriate category, where composition of morphisms describes the assembly of systems into larger or more complicated systems.  Earlier works have used span and cospan categories to study the composition of physical systems.  For example, Baez and Pollard \cite{Baez-Poll} used cospans to study reaction networks.  Haugseng \cite{Hau} used spans to study classical topological field theories.   A span in a category $\C$ is a pair of morphisms in $\C$ with the same source.  B\'enabou \cite[p.\;22]{Ben} credits Yoneda \cite{Yon} with introducing the notion of a span in the category of categories. B\'enabou \cite{Ben} proved that if $\C$ is a category with pullbacks, then there is a bicategory, ${\rm Span}\!\(\C\)$, whose objects, morphisms, and $2$-morphisms are the respective objects, spans, and maps of spans in $\C$.  View this bicategory as a category, a \emph{span category}, by ignoring the bicategory structure and taking isomorphism classes of spans in $\C$, that Section~\ref{sec:physsys} defines, as the morphisms.  Leinster \cite[p.\,26]{Lein} discussed bicategories and B\'enabou's work.\nobreak

Denote by $\Set$ the category whose objects are sets, whose morphisms are functions between sets, and where composition of functions defines the composition of morphisms.  Suppose that $X$, $Y$, and $Z$ are sets and that $f$ and $g$ are functions that respectively map $X$ and $Y$ to the set $Z$.  Take $\rho_X$ and $\rho_Y$ to be the canonical projections \[\rho_X\colon X\times Y \to X\quad {\rm and}\quad \rho_Y\colon X\times Y \to Y,\] and take $\pi_X$ and $\pi_Y$ to be the respective restrictions of $\rho_X$ and $\rho_Y$ to the fibered product $X\times_Z Y$, the subset of $X\times Y$ consisting of all elements on which $f$ is equal to $g$.  Section~\ref{sec:spantight} discusses the universal properties of the fibered product in $\Set$.  These properties give a prescription for composing certain spans in $\Set$ in a way that seems strikingly similar to the way in which classical mechanical systems appear to compose.   

Baez initiated the current line of research by proposing that the study of classical mechanics might have a foundation in category theory.  In particular, he suggested that classical mechanical systems could be morphisms in an appropriate span category, where composition of morphisms using fibered products would describe the composition of the systems \cite{JBPC}.  However, modeling classical mechanical systems necessitates working with spans in categories other than $\Set$, where the fibered product lacks the universal properties that it has in $\Set$. The current paper provides a way to circumvent this technical obstruction by generalizing the construction of a span category.  Generalized span categories differ from span categories in the way that they define composition.  The composite of spans may not be defined in a given category $\C$ because composition requires the existence of pullbacks and $\C$ may not have pullbacks.  The rough idea is to map $\C$ by a functor $\F$ into another category $\C^\prime$.  The composite of two spans in $\C$ will be a span in $\C$ with the property that its image under $\F$ is a composite in $\C^\prime$ of the the images under $\F$ of the spans being composed in $\C$.  If the functor $\F$ satisfies a certain property, namely that it is span tight, then composition will be well defined for isomorphism classes of spans in $\C$.

Section~\ref{sec:physsys} defines a span, a cospan, and an isomorphism class of spans.  The language it employs is arguably nonstandard from a category theorist's perspective but we have found it helpful both for presenting the results to nonspecialists in category theory and for our specific applications.  Section~\ref{sec:nopb} proves that some categories that are important in the study of classical mechanics do not have pullbacks, which leads to the present goal of generalizing the definition of a span category. Denote by $\Dif$ the category whose objects are smooth manifolds and whose morphisms are smooth functions.  Since two submanifolds of a given manifold may not intersect transversally, the fibered product of two manifolds is not necessarily an embedded submanifold of their product.  For reasons related to this, $\Dif$ does not have pullbacks. This technical difficulty that Spivak \cite{Spiv2} encountered parallels a central technical difficulty of the current paper.

Suppose that $\C$ and $\C^\prime$ are categories and $\F$ is a functor from $\C$ to $\C^\prime$.  Section~\ref{sec:spantight} defines an $\F$-pullback of a cospan in $\C$ and the span tightness of the functor $\mathcal F$, as well as the composite of two spans along an $\F$-pullback.  While the notion of an $\F$-pullback generalizes the notion of a pullback in a way that is sufficient for the current setting, without an additional condition on $\mathcal F$ it is not enough to provide a method for composing isomorphism classes of spans.  Section~\ref{sec:genSpan} proves that if the functor $\mathcal F$ is span tight, then there exists a category ${\rm Span}(\C,\F)$ whose objects are the objects in $\C$ and whose morphisms are isomorphism classes of spans in $\C$. Composition in this generalized span category is defined using $\F$-pullbacks and appears to depend on the functor $\F$.   In \cite{BWY}, we apply the tools that we develop to the study of classical mechanics.  Section~\ref{sec:app} briefly discusses some examples of generalized span categories.

\subsection*{Acknowledgements}

We thank Professor John C. Baez for initially suggesting this project to us, for his valuable guidance, and for his enlightening instruction.  We thank him as well for helping us to understand his motivation for this project and the orientation of this current project in his larger research program. We thank Professors Daniel Cicala and Christina Vasilakopoulou for many insightful discussions on category theory.  

We dedicate this work to Professor V.\,S.\,Varadarajan, a dear mentor and friend, whose memory we will always treasure.

\section{Span Categories}\label{sec:physsys}

\subsection{Spans and their Isomorphism Classes}

A \emph{span} in a category $\C$ is a pair of morphisms in $\C$ with the same source and a \emph{cospan} in $\C$ is a pair of morphisms in $\C$ with the same target. For any span $S$ in $\C$, write \[S = \(\sL, \sR\),\] where $\SL$, $\SR$, and $\SA$ are objects in $\C$, \[\sL\colon \SA\to \SL, \quad {\rm and} \quad \sR\colon \SA\to \SR.\] Utilize the same notation if $S$ is a cospan, but where $s_L$ and $s_R$ map $\SL$ and $\SR$ to $\SA$, respectively.   The span $S$ and the cospan $C$ are these respective diagrams:

\medskip

\begin{center} 
\begin{tikzpicture}[->=stealth',node distance=2.25cm, auto, scale = .8]
\node(K) at (-2,0) {$\SL$};
\node(L) at (2,0) {$\SR$};
\node(M) at (0,2) {$\SA$};
\draw[->] (M) to node[yshift = 6, inner sep = 2, left] {$\sL$} (K);
\draw[->] (M) to node[yshift = 6, inner sep = 2, right] {$\sR$} (L);
\end{tikzpicture}
\hspace{.75in}
\begin{tikzpicture}[->=stealth',node distance=2.25cm, auto, scale = .8]
\node(K) at (-2,2) {$\CL$};
\node(L) at (2,2) {$\CR$};
\node(M) at (0,0) {$\CA$};
\draw[->] (K) to node[yshift = -6, inner sep = 2, left] {$\cL$}(M);
\draw[->] (L) to node[yshift = -6, inner sep = 2, right] {$\cR$}(M);
\end{tikzpicture}
\end{center} 
\noindent For any span or cospan $S$ in $\C$, refer to the objects $\SA$, $\SL$, and $\SR$ in $\C$ as the \emph{apex}, \emph{left foot}, and \emph{right foot} of $S$, respectively.  

\begin{definition}
A span $S$ in $\C$ is \emph{paired with} a cospan $C$ in $\C$ if \[\CL = \SL, \quad \CR = \SR,\quad {\rm and} \quad \cL\circ \sL=\cR\circ \sR.\] 
\end{definition}

\begin{definition}
Take $S$ and $Q$ to be spans in $\C$ with $\SL$ equal to $\QL$ and $\SR$ equal to $\QR$.  A \emph{span morphism in $\C$ from $S$ to $Q$} is a morphism $\Phi$ in $\C$ from $\SA$ to $\QA$ with \[\sL = \qL\circ\Phi \quad {\rm and} \quad \sR = \qR\circ\Phi.\]  A \emph{span isomorphism in $\C$ from $S$ to $Q$} is a span morphism that is additionally an isomorphism.
\end{definition}

View a span $S$ paired with a cospan $C$ and view a span morphism $\Phi$ from a span $S$ to a span $Q$, respectively, as these commutative diagrams:

\medskip
\begin{center} 
\begin{tikzpicture}[->=stealth',node distance=2.25cm, auto, scale = .8]
\node(K) at (-2,0) {$\CL = \SL$};
\node(L) at (2,0) {$\CR = \SR$};
\node(M) at (0,2) {$\SA$};
\draw[->] (M) to node[yshift = 6, inner sep = 2, left] {$\sL$} (K);
\draw[->] (M) to node[yshift = 6, inner sep = 2, right] {$\sR$} (L);
\node(MM) at (0,-2) {$\CA$};
\draw[->] (K) to node[yshift = -6, inner sep = 2, left] {$\cL$}(MM);
\draw[->] (L) to node[yshift = -6, inner sep = 2, right] {$\cR$}(MM);
\end{tikzpicture}
\hspace{.5in}
\begin{tikzpicture}[->=stealth',node distance=2.25cm, auto, scale = .8]
\node(L) at (-2,0) {$\SL = \QL$};
\node(R) at (2,0) {$\SR = \QR$};
\node(T) at (0,2) {$\SA$};
\node(B) at (0,-2) {$\QA$};
\draw[->] (T) to node[yshift = 6, inner sep = 2, left] {$\sL$} (L);
\draw[->] (T) to node[yshift = 6, inner sep = 2, right] {$\sR$} (R);
\draw[->] (B) to node[yshift = -6, inner sep = 2, left] {$\qL$} (L);
\draw[->] (B) to node[yshift = -6, inner sep = 2, right] {$\qR$} (R);
\draw[->] (T) to node[inner sep = 3, left] {$\Phi$} (B);
\end{tikzpicture}
\end{center} 

\begin{proposition}\label{prop:spammap}
For any span isomorphism $\Phi$, the inverse $\Phi^{-1}$ is also a span isomorphism.  Furthermore, a composite of span morphisms is a span morphism.
\end{proposition}

\subsection{Pullbacks in a Category $\C$}

This subsection introduces the notion of a pullback of a cospan and defines what it means for a category to have pullbacks.

\begin{definition}\label{def:2:uniqespanmorph}
A span $S$ in $\C$ is a \emph{pullback of a cospan $C$ in $\C$} if it is paired with $C$ and if for any other span $Q$ in $\C$ that is also paired with $C$, there exists a unique span morphism $\Phi$ in $\C$ from $Q$ to $S$.
\end{definition}

The morphism $\Phi$ in Definition~\ref{def:2:uniqespanmorph} is the unique span morphism so that this diagram commutes:

\begin{center} 
\begin{tikzpicture}[->=stealth',node distance=2.25cm, auto, scale = .9]
\node(W) at (0,0) {$\CL$};
\node(X) at (0,2) {$\SA$};
\node(Y) at (2.5,2) {$\CR$};
\node(Z) at (2.5,0) {$\CA$};
\node(R) at (-1.65,3.32) {$\QA$};
\draw[->] (X) to node[inner sep = 5, above] {$\sR$}(Y);
\draw[->] (X) to node[inner sep = 5, left] {$\sL$}(W);
\draw[->] (Y)to node[inner sep = 5, right] {$\cR$}(Z);
\draw[->] (W)to node[inner sep = 5, below] {$\cL$}(Z);
\draw[->, densely dashed] (R) to (X);
\node[] at (-.48, 2.9) {$\Phi$};
\draw[->,bend right] (R) to node[inner sep = 5, below]{\phantom{$\displaystyle\int$}$\qL$\phantom{$\displaystyle\int$}}(W);
\draw[->,bend left](R)to node[inner sep = 5, above]{$\qR$}(Y);
\end{tikzpicture}
\end{center}

\begin{definition}\label{def:hasPB}
A category $\C$ \emph{has pullbacks} if for any cospan $C$ in $\C$, there is a span $S$ in $\C$ that is a pullback of $C$ and $S$ is unique up to a span isomorphism in $\C$.
\end{definition}

The diagram consisting of a pullback $S$ of a cospan $C$ together with $C$ is a \emph{pullback square}.  We have found it useful to separately define the parts of a pullback square.

\subsection{Examples of Categories that have Pullbacks}  Denote by $\Top$ the category whose objects are topological spaces and whose morphisms are continuous functions.  The categories $\Set$ and $\Top$ are examples of categories that have pullbacks, as MacLane \cite{Mac} discusses and Awodey \cite{Awo} discusses more specifically for $\Set$.  We provide a proof here for the convenience of the reader.

Let $C$ be a cospan in $\Set$, and let $\rho_L$ and $\rho_R$ be the canonical projections \[\rho_L\colon \CL\times\CR \to \CL\quad {\rm and} \quad \rho_R\colon \CL\times\CR \to\CR.\]  Denote by $\SA$ the fibered product \[C_L\times_{\CA}\CR := \{(x,y)\in \CL\times\CR\colon (\cL\circ \rho_L)(x,y) = (\cR\circ \rho_R)(x,y)\}.\]   Take $\SL$ and $\SR$ to be equal to $\CL$ and $\CR$, respectively, and let $\sL$ and $\sR$ be the respective restrictions of $\rho_L$ and $\rho_R$ to the set $\SA$.  For any span $P$ that is paired with $C$, take $\Phi$ to be the function \[\Phi\colon \PA \to \CL\times \CR \quad{\rm by}\quad a\mapsto(\pL(a), \pR(a))\quad (\forall a\in \PA).\]  The function $\Phi$ is the unique function from $\PA$ to $\CL\times \CR$ such that \begin{equation}\label{eq:2:spanmorphinset}\pL = \rho_L\circ\Phi \quad {\rm and}\quad \pR = \rho_R\circ \Phi.\end{equation} The image of $\Phi$ is contained in $\SA$ and so $\Phi$ is a span morphism from $P$ to $S$.  Since any other span morphism from $P$ to $S$ defines a function from $P$ to $\CL\times \CR$ with the property given by \eqref{eq:2:spanmorphinset}, the function $\Phi$ is the unique span morphism from $P$ to $S$, and so the span $S$ is  a pullback of the cospan $C$.  

Suppose that $C$ is a cospan in $\Top$, and let $\rho_L$ and $\rho_R$ again be the canonical projections on $\CL\times \CR$. The product $\CL\times \CR$ with the product topology is a topological space. The fibered product $\SA$ given above is a subset of $\CL\times \CR$ and is a topological space with the subspace topology.  The projections $\sL$ and $\sR$ are continuous maps, and so $(\sL, \sR)$ is a pullback of $C$.  The proof of this fact is nearly the same as the proof in the setting of $\Set$, with only an additional straightforward check that the mappings involved are continuous.

\subsection{The Category ${\rm \bf Span}(\C)$} Suppose that $\C$ is a category that has pullbacks.  Suppose that $[S]$ and $[Q]$ are isomorphism classes of spans with respective representatives $S$ and $Q$, and $\SR$ is equal to $\QL$.  Since $\C$ has pullbacks, there is a span $P$ that is a pullback of the cospan $(\sR, \qL)$. Define by $[(\sL\circ\pL, \qR\circ\pR)]$ the composite $[S]\circ[Q]$.  Take the objects in $\C$ to be the objects in ${\rm Span}(\C)$, the isomorphism classes of spans in $\C$ to be the morphisms in ${\rm Span}(\C)$, and $\SR$ and $\SL$ to respectively be the source and target of the span $[S]$.  Given an object $X$ in $\C$ and the identity morphism ${\rm Id}$ taking $X$ to $X$, define by $[({\rm Id}, {\rm Id})]$ the identity morphism in ${\rm Span}(\C)$ with $X$ as both source and target.  It is well known that ${\rm Span}(\C)$ is a category \cite{Ben}.  Our treatment in Section~\ref{sec:spantight} of generalized span categories specializes in the case when $\C$ has pullbacks to give a proof that ${\rm Span}(\C)$ is a category.  If $\C$ does not have pullbacks, then the existence of $P$ is not guaranteed.  The next section will demonstrate that some categories that are important in classical mechanics, and more generally in differential geometry, do not have pullbacks.

\section{Some Categories that do not have Pullbacks}\label{sec:nopb}

\subsection{Background}

Refer to \cite{Lee} as a standard reference for smooth manifold theory.  The current section reviews some well-known definitions in order to explicitly establish language and notational conventions. 

An $m$-dimensional manifold is a triple $(M, \mathcal T_M, \mathcal A_M)$ such that:
\begin{enumerate}
\item[(1)] $M$ is a set; 
\item[(2)] $\mathcal T_M$ is a topology for $M$ that is Hausdorff and second countable; 
\item[(3)] $\mathcal A_M$ is an atlas, a collection of homeomorphisms such that the domain of each element of $\mathcal A_M$ is an open subset of $M$, the collection of domains of the elements of $\mathcal A_M$ form an open cover for $M$, and the range of each element of $\mathcal A_M$ is an open subset of $\mathds R^m$.  
\end{enumerate}
If $\mathcal A_M$ is maximal with respect to the property that for any $\phi$ and $\psi$ in $\mathcal A_M$ that have intersecting domains, the \emph{transition function} $\phi\circ \psi^{-1}$ and its inverse are of class $C^r$ ($r$-times continuously differentiable), then $M$ is a \emph{$C^r$\!--\,manifold}. Only the smooth case, when $r$ is infinity, is relevant to this paper.  Refer to the elements of $\mathcal A_M$ as \emph{coordinates} and refer to their domains as \emph{charts}.  It is customary to denote by $M$ a manifold $(M, \mathcal T_M, \mathcal A_M)$, and we generally follow this convention, except when it is important to explicitly distinguish between the manifold, the topological space associated to the manifold, and the underlying set.  Reference to the manifold $M$, the topological space $M$, and the underlying set $M$, will respectively be a reference to the triple $(M, \mathcal T_M, \mathcal A_M)$, the pair $(M, \mathcal T_M)$, and the set $M$.  Denote by $\Dif$ the category whose objects are smooth manifolds and whose morphisms are smooth maps between smooth manifolds.

\subsection{Some Functors that preserve Pullbacks}
 
 Suppose that $\C$ is a locally small category and that $X$ is an object in $\C$.  Denote by $\Hom(X,-)$ the \emph{hom functor} from $\C$ to $\Set$.  
 A functor $\F$ that takes $\C$ to $\Set$ is \emph{representable} if there is an object $B$ in $\C$ so that $\F$ is naturally isomorphic to $\Hom(B,-)$.
 
 The categories $\Dif$, $\Top$, and $\Set$ are locally small and there are forgetful functors, each to be ambiguously denoted by $\F$, from $\Dif$ to $\Top$ and from $\Top$ to $\Set$ given by \[\F(M, \mathcal T_M, \mathcal A_M) = (M, \mathcal T_M) \quad {\rm and}\quad \F(M, \mathcal T_M) = M.\] The morphisms in $\Dif$ and $\Top$ are entirely determined by their action on the underlying sets and so the forgetful functor in each case maps a given source category to a subcategory of the target category.  The functor obtained by composing  the above forgetful functors is the forgetful functor, denoted again by $\F$, from $\Dif$ to $\Set$.

 A functor $\F$ from a category $\C$ to a category $\C^\prime$ \emph{preserves pullbacks} if for any cospan $C$ in $\C$, if $S$ is a pullback of $C$, then $\F(S)$ is a pullback of $\F(C)$. The following lemma is a special case of a more general result that guarantees that representable functors preserve pullbacks \cite[p.~64]{Bor}.  The proof of Lemma~\ref{3:lem:reppreslim} is included for the convenience of the reader, because the language in our definition of a pullback is slightly different from Borceux's \cite{Bor}.

\begin{lemma}\label{3:lem:reppreslim}
For any locally small category  $\C$ and any object $B$ in $\C$, the functor $\Hom(B,-)$ preserves pullbacks, where \[\Hom(B,-)\colon \C\to\Set.\]
\end{lemma}

\begin{proof}
Take $X$ and $Y$ to be objects in $\C$.   For any morphism $f$ in $\C$ from $X$ to $Y$, denote by $\tilde{f}$ the morphism $\Hom(B,f)$ that is defined to act on any $\beta$ in $\Hom(B,X)$ by \[\tilde{f}(\beta) = f\circ \beta.\]  Suppose that $C$ is a cospan in $\C$ and that $S$ is a pullback of $C$.   Since $\C$ is locally small, the functor $\Hom(B,-)$ maps the cospan $C$ to a cospan $\Hom(B,C)$ in $\Set$, taking the pair $(\cL, \cR)$ to the pair $(\tcL, \tcR)$.  It similarly maps the span $S$ to the span $\Hom(B, S)$. Since $S$ is a pullback of $C$, it is paired with $C$ and so \[\cL\circ\sL = \cR\circ \sR.\] Therefore, for any $\psi$ in $\Hom(B,\SA)$, \[(\tcL\circ\tsL)(\psi) = \cL\circ\sL\circ \psi = \cR\circ \sR\circ \psi = (\tcR\circ\tsR)(\psi),\] and so the span $\Hom(B, S)$ is paired with the cospan $\Hom(B, C)$.

Denote respectively by $\rho_L$ and $\rho_R$ the canonical projections from $\Hom(B,\CL)\times\Hom(B,\CR)$ to $\Hom(B,\CL)$ and $\Hom(B,\CR)$, and by $\QA$ the set \begin{align*}&\Hom(B,\CL)\times_{\Hom(B,\CA)}\Hom(B,\CR) \\&\qquad\qquad = \left\{\alpha\in \Hom(B,\CL)\times\Hom(B,\CR)\colon (\tcL\circ \rho_L)(\alpha) = (\tcR\circ \rho_R)(\alpha)\right\}.\end{align*}  Take $\qL$ and $\qR$ to be the respective restrictions of $\rho_L$ and $\rho_R$ to  $\QA$.  Denote by $Q$ the span $(\qL, \qR)$ in $\Set$, a pullback of the cospan $\Hom(B,C)$.

For any $\alpha$ in $\QA$, there are morphisms $\alpha_L$ and $\alpha_R$ in $\C$ that respectively map $B$ to $\CL$ and $\CR$, and $\alpha$ is equal to $(\alpha_L, \alpha_R)$.  Furthermore, \[\cL\circ\alpha_L = \tcL(\alpha_L) = (\tcL\circ \qL)(\alpha) = (\tcR\circ \qR)(\alpha)= \tcR(\alpha_R) = \cR\circ\alpha_R.\] The pair $(\alpha_L, \alpha_R)$ is, therefore, a span in $\C$ that is paired with $C$ and, since $S$ is a pullback of $C$, there is a unique span morphism $\phi_\alpha$ in $\C$ from $(\alpha_L, \alpha_R)$ to $S$ that maps $B$ to $\SA$.  Take $\Phi$ to be the function from $Q$ to $\Hom(B, S)$ that is defined for each $\alpha$ in $\QA$ by \[\Phi(\alpha) = \phi_\alpha.\] The morphism $\phi_\alpha$ is a span morphism, implying that \[\sL\circ\phi_\alpha =  \alpha_L \quad {\rm and}\quad \sR\circ\phi_\alpha =  \alpha_R.\] These equalities further imply that \[(\tsL\circ\Phi)(\alpha) = \sL\circ\phi_\alpha, \quad  \alpha_L  = \qL(\alpha), \quad (\tsR\circ\Phi)(\alpha)= \sR\circ\phi_\alpha, \quad{\rm and} \quad \alpha_R = \qR(\alpha),\] and so \[(\tsL\circ\Phi)(\alpha) = \qL(\alpha) \quad {\rm and}\quad (\tsR\circ\Phi)(\alpha) = \qR(\alpha).\]  The morphism $\Phi$ in $\Set$ is, therefore, a span morphism and is unique since $\phi_\alpha$ is uniquely determined.  Since $Q$ is a pullback of $\Hom(B,C)$, the span $\Hom(B,S)$ is as well, and so $\Hom(B,-)$ maps pullbacks in $\C$ to pullbacks in $\Set$.
\end{proof}

Take $(\bf1, \mathcal T_{\bf1}, \mathcal A_{\bf1})$ to be the singleton manifold with its unique discrete topology.  Denote by $\Dif((\bf1, \mathcal T_{\bf1}, \mathcal A_{\bf1}),-)$ the functor that maps any given manifold $(X,\mathcal T_X, \mathcal A_X)$ to the set of functions $\Dif(({\bf1}, \mathcal T_{\bf1}, \mathcal A_{\bf1}),(X,\mathcal T_X, \mathcal A_X))$.  Compress notation by writing $\Dif({\bf1},X)$ instead of $\Dif(({\bf1}, \mathcal T_{\bf1}, \mathcal A_{\bf1}),(X,\mathcal T_X, \mathcal A_X))$ and $\Dif({\bf1},-)$ instead of $\Dif((\bf1, \mathcal T_{\bf1}, \mathcal A_{\bf1}),-)$.   Lemma~\ref{3:lem:nat} is well known, but we include it for the reader's convenience.  

\begin{lemma}\label{3:lem:nat}
The forgetful functor $\F$ from $\Dif$ to $\Set$ is naturally isomorphic to the functor $\Dif({\bf1},-)$.
\end{lemma}

\begin{proof}

For any object $(X,\mathcal T_X, \mathcal A_X)$ in $\Dif$, take $\eta_X$ to be the function from $\Dif({\bf1},X)$ to $X$ that is given for any $h$ in $\Dif({\bf1},X)$ by 
\[\eta_X(h) = h({\bf1}).\]  The function $\eta_X$ is invertible, with inverse $\eta_X^{-1}$ that is given for any $x$ in $X$ by \[\eta_X^{-1}(x) = h_x\colon{\bf1}\mapsto x,\] hence $\eta_X$ is an isomorphism.

Suppose that $g$ is in $\Dif({\bf1},X)$ and $f$ is a smooth function from $(X,\mathcal T_X, \mathcal A_X)$ to $(Y,\mathcal T_Y, \mathcal A_Y)$.  The point $\eta_{X}(g)$ in $X$ is equal to $g({\bf1})$ and $\F(f)(g({\bf1}))$ is equal to $f(g({\bf1}))$.  Furthermore, $\Dif({\bf1},f)$ takes $g$ to $f\circ g$ and $\eta_{Y}$ takes $(f\circ g)$ to $(f\circ g)({\bf1})$, which is equal to $f(g({\bf1}))$, and so this diagram is commutative:

\medskip

\begin{center}
\begin{tikzpicture}[node distance=2.25cm, auto]
\node(W) at (0,0) {$\Dif({\bf1},Y)$};
\node(X) at (0,2) {$\Dif({\bf1},X)$};
\node(Y) at (4,2) {$\F(X,\mathcal T_X, \mathcal A_X)$};
\node(Z) at (4,0) {$\F(Y,\mathcal T_Y, \mathcal A_Y)$};

\draw[->] (X) to node[above]{$\eta_{X}$}(Y);
\draw[->] (X) to node[left]{$\Dif({\bf1},f)$}(W);
\draw[->] (Y)to node[right] {$\F(f)$}(Z);
\draw[->] (W)to node[below] {$\eta_{Y}$}(Z);
\end{tikzpicture}
\end{center}

\noindent   The functor $\Dif({\bf1},-)$ is, therefore, naturally isomorphic to $\F$.
\end{proof}

Lemma~\ref{3:lem:reppreslim} and Lemma~\ref{3:lem:nat} together imply Propostion~\ref{3:prop:repfun}.

\begin{proposition}\label{3:prop:repfun}
The forgetful functor $\F$ from $\Dif$ to $\Set$ preserves pullbacks.
\end{proposition}

\subsection{$\SurjSub$ does not have Pullbacks}

Denote by $\SurjSub$ the category whose objects are smooth manifolds and whose morphisms are smooth surjective submersions.  This category is important in the study of classical mechanical systems because a map that takes the configuration space of a classical mechanical system to the configuration space of a subsystem should be a surjective submersion.  The category $\SurjSub$ is an example of a category that does not have pullbacks.  To see this, denote the one and two point manifolds respectively by $\mathbf{1}$ and $\mathbf{2}$.  Take $f$ to be the unique map from $\mathbf{2}$ to $\mathbf{1}$.  Denote by ${\rm Id}$ the identity map from $\mathbf{2}$ to $\mathbf{2}$.  The span $({\rm Id}, {\rm Id})$ is paired with the cospan $(f,f)$.  

Suppose that $\pi_L$ and $\pi_R$ are the canonical projections from $\mathbf{2\times_1 2}$ to $\mathbf 2$.  Suppose that $S$ is a pullback of $(f,f)$ in $\SurjSub$. Proposition~\ref{3:prop:repfun} implies that the image of $S$ under the forgetful functor from $\Dif$ to $\Set$ is the span $(\pi_L, \pi_R)$ and that there is a unique $\Phi$ in $\SurjSub$ so that this diagram commutes:

\medskip

\begin{center} 
\begin{tikzpicture}[->=stealth',node distance=2.25cm, auto, scale = .9]
\node(W) at (0,0) {$\mathbf{2}$};
\node(X) at (0,2) {$\mathbf{2}\times_{\mathbf{1}}\mathbf{2}$};
\node(Y) at (2.5,2) {$\mathbf{2}$};
\node(Z) at (2.5,0) {$\mathbf{1}$};
\node(R) at (-1.65,3.32) {$\mathbf{2}$};
\draw[->] (X) to node[inner sep = 5, above] {$\pi_R$}(Y);
\draw[->] (X) to node[inner sep = 5, left] {$\pi_L$}(W);
\draw[->] (Y)to node[inner sep = 5, right] {$f$}(Z);
\draw[->] (W)to node[inner sep = 5, below] {$f$}(Z);
\draw[->, densely dashed] (R) to (X);
\node[] at (-.48, 2.9) {$\Phi$};
\draw[->,bend right] (R) to node[inner sep = 5, below]{\phantom{$\displaystyle\int$}${\rm Id}$\phantom{$\displaystyle\int$}}(W);
\draw[->,bend left](R)to node[inner sep = 5, above]{${\rm Id}$}(Y);
\end{tikzpicture}
\end{center}

However, $\mathbf{2}\times_{\mathbf{1}}\mathbf{2}$ is isomorphic to $\mathbf{2}\times\mathbf{2}$, a set with four elements, and so $\Phi$ is a surjection from a set with only two elements to a set with four elements, which is not possible.  The cospan $(f,f)$ does not have a pullback in $\SurjSub$, and so $\SurjSub$ does not have pullbacks.

The technical obstructions that arise in working with categories of smooth manifolds with surjective morphisms are elementary, as the above example indicates. To avoid any misunderstanding, notice that examples of cospans that do not have pullbacks may be constructed in a similar way that involve connected manifolds with nontrivial smooth structures.  For example, take $f$ to be a projection from $\mathds R^4$ to $\mathds R^2$.  Take $\pi_L$ and $\pi_R$ to be the restrictions to the subset $\mathds R^4 \times_{\mathds R^2} \mathds R^4$ of the respective projections $\rho_L$ and $\rho_R$ that are defined for each $x$ and $y$ in $\mathds R^4$ by \[\rho_L(x,y) = x \quad {\rm and}\quad \rho_R(x,y) = y.\]  Take ${\rm Id}$ to be the identity map from $\mathds R^4$ to $\mathds R^4$.  The span $({\rm Id}, {\rm Id})$ is paired with $(f,f)$, but there is no smooth surjection from $\mathds R^4$ to the fibered product since the fibered product is six-dimensional.  This counterexample is relevant to demonstrating that the category of Riemannian manifolds with surjective Riemannian submersions and the category of symplectic manifolds with surjective Poisson maps also do not have pullbacks.  These categories play a fundamental role in \cite{BWY}.

\subsection{$\Dif$ does not have Pullbacks}

Suppose throughout this subsection that $f$ and $g$ are morphisms in $\Dif$ that have mutual target $(Z,\mathcal T_Z, \mathcal A_Z)$ and respective sources $(X,\mathcal T_X, \mathcal A_X)$ and $(Y,\mathcal T_Y, \mathcal A_Y)$.  Recall that $\pi_X$ and $\pi_Y$ are the respective projections from the set $X\times_ZY$ to $X$ and $Y$.  Take ${\mathcal T}_{S}$ to be the subspace topology on $X\times_ZY$ that $X\times_ZY$ inherits from the product topology on $X\times Y$ and with respect to which $\pi_X$ and $\pi_Y$ are both continuous.  View the functions $f$ and $g$ as functions in $\Top$ that have the topological space $(Z, \mathcal T_Z)$ as their mutual target and the topological spaces $(X,\mathcal T_X)$ and $(Y,\mathcal T_Y)$ as their respective sources. Suppose that $(W, \mathcal T_W, \mathcal A_W)$ is an embedded submanifold of $(Z, \mathcal T_Z, \mathcal A_Z)$.  Refer to \cite[p. 143-144]{Lee} for further discussion of transversality and, in particular, for the proof of Proposition~\ref{prop:trans}.

\begin{definition}\label{Def:Trans}
The smooth function $f$ is \emph{transverse to $W$} if, for every $x$ in $f^{-1}(W)$, the spaces $T_{f(x)}W$ and $\d f\!\(T_{x}X\)$ together span $T_{f(x)}Z$.  The smooth functions $f$ and $g$ are \emph{transverse} if, for every point $x$ in $X$ and $y$ in $Y$ with $f\!\(x\)$ and $g\!\(y\)$ both equal to $z$, \[\d f\!\(T_xX\) + \d g\!\(T_y Y\) = T_z Z.\]
\end{definition} 

\begin{proposition}\label{prop:trans}
Suppose that $X$ and $Z$ are smooth manifolds and $W$ is an embedded submanifold of $Z$. If $f$ is a smooth map from $X$ to $Z$ that is transverse to $W$, then $f^{-1}(W)$ is an embedded submanifold of $X$ whose codimension is equal to the codimension of $W$ in $Z$.
\end{proposition} 

\begin{proposition}\label{prop:TransverseansFiberedProductA}
If $f$ and $g$ are transverse, then the fibered product $X\times_{Z}Y$ is a smooth embedded submanifold of codimension equal to the dimension of $Z$. Furthermore, the span $(\pi_X, \pi_Y)$ in $\Dif$ is a pullback of $(f,g)$.
\end{proposition}

\begin{proof}
\begin{pf}
Denote by $\Delta_Z$ the diagonal $\{\(z,z\)\colon z\in Z\}$ of $Z\times Z$, an embedded submanifold of $Z\times Z$. The function $f\times g$, with \[f\times g\colon X\times Y\to Z\times Z \quad \text{by} \quad (x,y)\mapsto (f(x),g(y)),\] is smooth and $\(f\times g\)^{-1}\!\(\Delta_{Z}\)$ is equal to $X\times_{Z} Y$. Since $f$ and $g$ are transverse, the function $f\times g$ is transverse to $\Delta_{Z}$.  Proposition~\ref{prop:trans} implies that $X\times_{Z} Y$ is a smooth manifold with codimension in $X\times Y$ equal to the dimension of $\Delta_{Z}$.  The dimension of $\Delta_{Z}$ is equal to that of $Z$, and so the codimension of $X\times_{Z} Y$ in $X\times Y$ is equal to the dimension of $Z$.

To show that $(\pi_X, \pi_Y)$ is a pullback of $(f,g)$, suppose that $S$ is a span in $\Dif$ that is paired with $(f,g)$.  Define for each $s$ in $\SA$ the span morphism $\Phi$ from $S$ to $(\pi_X, \pi_Y)$ by \[\Phi(s) = (\sL(s), \sR(s)).\]  Take $\Phi^\prime$ to be another span morphism from $S$ to $(\pi_X, \pi_Y)$.  For any $s$ in $\SA$, \[\pi_X(\Phi^\prime(s)) = \sL(s)\quad {\rm and}\quad \pi_Y(\Phi^\prime(s)) = \sR(s),\] and so $\Phi^\prime(s)$ is equal to $\Phi(s)$.  Therefore, the morphism $\Phi^\prime$ is equal to $\Phi$.  Since $\Phi$ is unique, $(\pi_X, \pi_Y)$ is a pullback.  
\end{pf}
\end{proof}
 
If $f$ and $g$ are in $\SurjSub$ with mutual target $Z$, then they are transverse and so Proposition~\ref{prop:TransverseansFiberedProductA} implies the following:

 \begin{proposition}\label{prop:TransverseansFiberedProductDim}
If $\(f, g\)$ is a cospan in $\SurjSub$, then the fibered product $X\times_{Z}Y$ is a smooth embedded submanifold of $X\times Y$ of dimension ${\rm dim}\!\(X\times_ZY\)$, where \[{\rm dim}\!\(X\times_ZY\) = {\rm dim}\!\(X\) +{\rm dim}\!\(Y\) - {\rm dim}\!\(Z\).\] 
\end{proposition}

For the following proposition, take $(f,g)$ to be a cospan in $\Dif$ but where the maps $f$ and $g$ are not assumed to be transverse: 

\begin{proposition}\label{Sandxzy}
If $S$ is a span in $\Dif$ that is a pullback of $(f, g)$, and if $(\pi_X, \pi_Y)$ and $(\sL, \sR)$ are span isomorphic as spans in $\Top$, then $X\times_Z Y$ has a manifold structure.
\end{proposition}

\begin{proof}
Take $\Phi$ to be the unique span isomorphism in $\Top$ from $S$ to $(\pi_X, \pi_Y)$.  The homeomorphism $\Phi$ transports the manifold structure of $\SA$ to $X\times_Z Y$, giving it a manifold structure as well.
\end{proof}

If $S$ is a span in $\Dif$ that is paired with $(f,g)$, then the map $\Phi$, that is defined for each $s$ in $\SA$ by \[\Phi(s) = (\sL(s), \sR(s)),\] is a smooth map from $\SA$ to $X\times Y$.  If $X\times_ZY$ is an embedded submanifold of $X\times Y$, then $\Phi$ is a smooth map from $\SA$ to $X\times_ZY$ and is the unique such map, implying the following proposition:

\begin{proposition}\label{prop:embedsubmanpullback}
If $X\times_ZY$ is an embedded submanifold of $X\times Y$, then $(\pi_X, \pi_Y)$ is a span in $\Dif$ and a pullback of $(f,g)$.
\end{proposition}

Propositions \ref{prop:TransverseansFiberedProductA} and \ref{prop:embedsubmanpullback} together imply the following proposition:
\begin{proposition}
If $(f,g)$ is a cospan in $\Dif$ and $f$ and $g$ are transverse, then $(\pi_X, \pi_Y)$ is a pullback of $(f,g)$ in $\Dif$.  
\end{proposition}

The following example demonstrates that $X\times_ZY$ may be a manifold and $(\pi_X, \pi_Y)$ may a pullback of a cospan $(f,g)$ in $\Dif$, but $X\times_ZY$ is not an embedded submanifold of $X\times Y$.  In light of Proposition~\ref{prop:TransverseansFiberedProductA}, such an example requires the functions $f$ and $g$ to be non-transverse.  It also shows that the forgetful functor from $\Dif$ to $\Top$ does not preserve pullbacks.

\begin{example}\label{EX:FPDiffTop}
Take $X$ and $Z$ to be equal to $\mathds R$ and take $Y$ to be the one point manifold $\{y_0\}$.  Suppose that $f$ is a smooth function from $X$ to $Z$, $(a_n)$ is a sequence in $\mathds R$ that converges to a point $a_0$ that is not equal to $a_n$ for any natural number $n$, and the zero set of $f$ is the set $\{a_0\}\cup\{a_n\colon n\in \mathds N\}$.  Take $g$ to be the zero function from $Y$ to $Z$.  The set $X\times_ZY$ is the subset $\{(a_0, y_0)\}\cup\{(a_n, y_0)\colon n\in \mathds N\}$ of $X\times Y$.

The span $(\pi_X, \pi_Y)$ is paired with the cospan $(f,g)$ in $\Top$, where $\pi_X$ and $\pi_Y$ are the restrictions to $X\times_ZY$ of the continuous projections $\rho_X$ and $\rho_Y$ that respectively take $X\times Y$ to $X$ and $Y$, where $X\times Y$ is endowed with the product topology.  The span $(\pi_X, \pi_Y)$ is a pullback of $(f,g)$, where $X\times_ZY$ is endowed with the subspace topology $\mathcal T_S$.  In this topology, each set $\{(a_n, y_0)\}$ is an open set, where $n$ varies over $\mathds N$.  Any open set containing $(a_0, y_0)$ contains infinitely many points.  

Now view  $(f,g)$ as a cospan in $\Dif$, so that $X\times_ZY$ is a manifold.  Each point in $X\times_ZY$ must be contained in a neighborhood that is homeomorphic to a point, and so $X\times_ZY$ must be endowed with the discrete topology $\mathcal T_D$. In this case, the manifold $X\times_ZY$ is not an embedded submanifold of $X\times Y$ since its topology is not the subspace topology.  Nevertheless, the span $(\pi_X, \pi_Y)$ is a pullback of $(f,g)$ in $\Dif$.

The topology on the image of the manifold $X\times_ZY$ under the forgetful functor from $\Dif$ to $\Top$ is $\mathcal T_D$, which is a finer topology than $\mathcal T_S$. The forgetful functor $\F$ from $\Dif$ to $\Top$ maps the pullback $(\pi_X, \pi_Y)$, where maps $\pi_X$ and $\pi_Y$ have the manifold $\left(X\times_ZY, \mathcal T_D, \mathcal A_{X\times_ZY}\right)$ as their common source, to the span $(\pi_X, \pi_Y)$, where the maps have the topological space $(X\times_ZY, \mathcal T_D)$ as their common source.  The identity map taking $(X\times_ZY, \mathcal T_D)$ to $(X\times_ZY, \mathcal T_S)$ is a continuous span morphism from $(\pi_X, \pi_Y)$ to $(\pi_X, \pi_Y)$, but the inverse is not continuous.  This demonstrates that the forgetful functor from $\Dif$ to $\Top$ does not preserve pullbacks.
\end{example}

The maps $f$ and $g$ in Example~\ref{EX:FPDiffTop} are not transverse because $f$ is smooth and constant on a sequence that is convergent to $a_0$, and so ${\rm d}f(T_{a_0}\mathds R)$ is zero-dimensional.  This example demonstrates that $f$ and $g$ may be non-transverse morphisms in $\Dif$, but $(f,g)$ nevertheless has a pullback in $\Dif$.  The forgetful functor $\F$ from $\Dif$ to $\Set$ preserves pullbacks, and so if $Q$ is a span in $\Dif$ and a pullback of $(f,g)$, then $\F(Q)$ is a span in $\Set$ that is a pullback of $(\F(f),\F(g))$.  Since $\Set$ has pullbacks, there is a span isomorphism in $\Set$ from $\F(Q)$ to $(\F(\pi_X), \F(\pi_Y))$.  This span isomorphism is only a bijection, and there should be no expectation that it preserves topological structure.  

The category $\Top$ also has pullbacks, and so if $(f,g)$ is a cospan in $\Top$, then a pullback of $(f,g)$ will exist and, in fact, the span $(\pi_X, \pi_Y)$ in $\Top$ is a pullback of $(f,g)$ where the maps $\pi_X$ and $\pi_Y$ have $(X\times_ZY,\mathcal T_S)$ as their common source.  Since the forgetful functor from $\Dif$ to $\Top$ does not preserve pullbacks, there is no guarantee that $Q$ being a pullback of $(f,g)$ implies that it is a pullback when mapped by a forgetful functor to $\Top$, as is the case in Example~\ref{EX:FPDiffTop}.  

The former discussion, and the example that precedes it, demonstrate that there is some subtlety involved in determining that $\Dif$ does not have pullbacks.  The proof of Proposition~\ref{Prop:NoDifPullback} presents an elementary counterexample. The functions given in the statement of Lemma~\ref{Lem:NoDifPullback} form this commutative diagram:

\medskip

\begin{center} 
\begin{tikzpicture}[->=stealth',node distance=2.25cm, auto, scale = .9]
\node(W) at (0,0) {$X$};
\node(X) at (0,2) {$X\times_Z Y$};
\node(Y) at (2.5,2) {$Y$};
\node(Z) at (2.5,0) {$Z$};
\node(R) at (-1.65,3.32) {$\QA$};
\draw[->] (X) to node[inner sep = 5, above] {$\pi_Y$}(Y);
\draw[->] (X) to node[inner sep = 5, left] {$\pi_X$}(W);
\draw[->] (Y)to node[inner sep = 5, right] {$g$}(Z);
\draw[->] (W)to node[inner sep = 5, below] {$f$}(Z);
\draw[->, densely dashed] (R) to (X);
\node[] at (-.48, 2.9) {$\Phi$};
\draw[->,bend right] (R) to node[inner sep = 5, below]{\phantom{$\displaystyle\int$}$\qL$\phantom{$\displaystyle\int$}}(W);
\draw[->,bend left](R)to node[inner sep = 5, above]{$\qR$}(Y);
\end{tikzpicture}
\end{center}

\begin{lemma}\label{Lem:NoDifPullback}
For any cospan $(f,g)$ in $\Dif$ and for any pullback $Q$ of $(f,g)$ in $\Dif$, there is a bijective span morphism $\Phi$ in $\Top$ from $\F(Q)$ to $(\pi_X, \pi_Y)$, where $\F$ is the forgetful functor from $\Dif$ to $\Top$, and $X\times_ZY$ is endowed with the subspace topology.
\end{lemma}

To clarify, Lemma~\ref{Lem:NoDifPullback} guarantees the existence of a bijective span morphism $\Phi$ in $\Top$, but $\Phi$ is not guaranteed to be an isomorphism in $\Top$ since its inverse need not be continuous.  

\begin{proof}
Take $Q$ to be a span in $\Dif$ that is a pullback of $(f,g)$. Define for each $a$ in $\QA$ the function $\Phi$ by \begin{equation}\label{III:Lem:Eq:Phi}\Phi(a) = (\qL(a), \qR(a)).\end{equation} The map $\Phi$ from $\QA$ to $X\times Y$ is smooth because the functions $\qL$ and $\qR$ are smooth.  The span $Q$ is paired with $(f,g)$, implying that the range of $\Phi$ is $X\times_Z Y$, and so $\Phi$ is a continuous function from $\QA$ to $X\times_Z Y$, where $X\times_ZY$ is endowed with the subspace topology.  Proposition~\ref{3:prop:repfun} implies that the forgetful functor ${\mathcal G}$ from $\Dif$ to $\Set$ preserves pullbacks, therefore ${\mathcal G}(\Phi)$ is a span morphism in $\Set$ from ${\mathcal G}(Q)$ to $(\pi_X, \pi_Y)$, where the pair of projections is viewed only as a pair of maps in $\Set$.  The span $Q$ is a pullback in $\Dif$, hence ${\mathcal G}(Q)$ is a span in $\Set$ that is a pullback of $(f,g)$, and so the map ${\mathcal G}(\Phi)$ is a bijection.  Maps between manifolds and topological spaces are determined by their behavior on the underlying sets, hence $\F(\Phi)$ is a continuous bijection.
\end{proof}

Although the fact that $\Dif$ does not have pullbacks is commonly cited in the literature, we found it difficult to locate a detailed proof of this fact and so present one here.

\begin{proposition}\label{Prop:NoDifPullback}
The category $\Dif$ does not have pullbacks.
\end{proposition}

\begin{proof}

Take $X$, $Y$, and $Z$ to be equal to $\mathds R$.  Take $f$ and $g$ to be the respective functions from $X$ to $Z$ and $Y$ to $Z$ given for each $x$ in $X$ and $y$ in $Y$ by \[f(x) = x^2 \quad {\rm and}\quad g(y) = y^2.\] Denote by ${\rm Id}$ the identity function from $\mathds R$ to $\mathds R$.  For any pullback $Q$ in $\Dif$ of $(f,g)$, since the span $({\rm Id}, {\rm Id})$ is paired with $(f,g)$, there is a span morphism $\Phi_1$ in $\Dif$ from $({\rm Id}, {\rm Id})$ to $Q$.  The composite $\qL\circ \Phi_1$  is equal to ${\rm Id}$, and so it is injective.  Therefore, $\Phi_1$ is injective.  Since $\Phi_1$ is a continuous injection from $\mathds R$ to $\QA$, the dimension of $\QA$ is at least 1.

The fibered product $X\times_Z Y$ is the set \[X\times_Z Y = \{(x,y)\colon |x| = |y|\}.\]  Take $\F$ to be the forgetful functor from $\Dif$ to $\Set$.  The span $(\pi_X, \pi_Y)$ is a pullback in $\Set$ of $(\F(f), \F(g))$.  There are intervals $I$ and $J$, each equal to either $(-\infty, 0)$ or $(0, \infty)$, so that the set $U$ is nonempty, where $U$ is given by \[U = \qL^{-1}(I)\cap \qR^{-1}(J).\]  Otherwise, the span $(0_X, 0_Y)$ would be a pullback in $\Dif$ of $(f,g)$, where $0_X$ and $0_Y$ are the respective zero functions from the one point manifold $\bf 1$ to $X$ and $Y$. Since $\F$ preserves pullbacks, there would then be a bijection between the underlying set $\bf 1$ and $X\times_ZY$.


The set $U$ is an open subset of $\QA$ as a non-empty intersection of open sets and is, therefore, a manifold.  The restricted functions $f\big|_I$ and $g\big|_J$ are surjective submersions onto $(0, \infty)$.  Proposition~\ref{prop:TransverseansFiberedProductDim} implies that the fibered product $I\times_{(0,\infty)}J$ is a one-dimensional manifold.  The span $(\pi_{I}, \pi_{J})$ is a pullback of $\big(f\big|_{I}, g\big|_{J}\big)$.  Since $\big(\qL\big|_U, \qR\big|_U\big)$ is paired with $\big(f\big|_{I}, g\big|_{J}\big)$, there is a span morphism $\Phi_2$ in $\Dif$ from $\big(\qL\big|_U, \qR\big|_U\big)$ to $(\pi_{I}, \pi_{J})$.  The function $\Phi_2$ is just the restriction of $\Phi$, given by \eqref{III:Lem:Eq:Phi}, to $U$.  Since $\Phi_2$ is an injection, $U$ has dimension at most 1, and so the dimension of $U$ is equal to 1.

The dimension of the manifold $\QA$ is also equal to 1 since $\QA$ contains $U$ as an open subset, and so $\QA$ is homeomorphic to either a line, an open interval, a half-open interval, or a circle, \cite{Gale}.  The subspace $X\times_ZY\setminus \{(0,0)\}$ of $X\times Y$ has four connected components, but the preimage $\Phi^{-1}(X\times_ZY\setminus \{(0,0)\})$ is the set $\QA$ with one point removed, and so has either one or two connected components. However, Lemma~\ref{Lem:NoDifPullback} implies that $\Phi$ is continuous.  Since $X\times_ZY\setminus \{(0,0)\}$ cannot have more connected components than its preimage, $Q$ cannot be a pullback of $(f,g)$, and so $(f,g)$ does not have a pullback in $\Dif$.%
\end{proof}


\section{Composition by $\F$-Pullbacks}\label{sec:spantight}

Assume, henceforth, that $\C$ and $\C^\prime$ are categories and that $\F$ is a functor from $\C$ to $\C^\prime$.  For any span $S$ in $\C$, denote by $\F(S)$ the span $(\F(\sL), \F(\sR))$ in $\C^\prime$.  For any cospan $C$ in $\C$, denote by $\F(C)$ the cospan $(\F(\cL), \F(\cR))$ in $\C^\prime$.

\subsection{$\F$-Pullbacks and Span Tightness}

\begin{definition}
The category $\C$ \emph{has $\F$-pullbacks in} $\C^\prime$ if for any cospan $C$ in $\C$, there is a span $S$ in $\C$ that is paired with $C$ and the span $\F(S)$ is a pullback of the cospan $\F(C)$ in $\C^\prime$.  In this case, the span $S$ is an \emph{$\F$-pullback} of $C$.
\end{definition}

Note that if $\C^\prime$ is equal to $\C$ and $\F$ is the identity functor, then an $\F$-pullback is simply a pullback.

\begin{definition}
Suppose that $S$ and $Q$ are spans in $\C$ such that: 
\begin{enumerate}
\item[(1)] $\SR = \QL$;
\item[(2)] there is a span $P$ in $\C$ that is a pullback of the cospan $(\sR, \qL)$.
\end{enumerate}
The \emph{composite of $S$ and $Q$ along $P$} is the span in $\C$ given by \[S\circ_P Q = (\sL\circ \pL, \qR\circ \pR).\] If $P$ is an $\F$-pullback, then the span $S\circ_P Q$ is an \emph{$\F$-pullback composite of $S$ and $Q$ along $P$}.
\end{definition}

These morphisms determine the construction of the composite span $S\circ_P Q$ along the $\F$-pullback $P$:

\medskip

\begin{center} 
\begin{tikzpicture}[->=stealth',node distance=2.25cm, auto, scale=.8]
\node(SL) at (-1.5,0) {$\SL$};
\node(SR) at (1.5,0) {$\SR = \QL$};
\node(SA) at (0,2) {$\SA$};
\draw[->] (SA) to node[yshift = 6, inner sep = 2, left] {$\sL$} (SL);
\draw[->] (SA) to node[yshift = 6, inner sep = 2, right] {$\sR$} (SR);

\node(QR) at (4.5,0) {$\QR$};
\node(RA) at (3,2) {$\QA$};
\draw[->] (RA) to node[yshift = 6, inner sep = 2, left] {$\qL$} (SR);
\draw[->] (RA) to node[yshift = 6, inner sep = 2, right] {$\qR$} (QR);

\node(QA) at (1.5,4) {$P_A$};
\draw[->] (QA) to node[yshift = 6, inner sep = 2, left] {$\pL$} (SA);
\draw[->] (QA) to node[yshift = 6, inner sep = 2, right] {$\pR$} (RA);

\end{tikzpicture}
\end{center}

\noindent The composite span $S\circ_P Q$ is this diagram:

\medskip

\begin{center}

\begin{tikzpicture}[->=stealth',node distance=2.25cm, auto, scale=.8]

\node(SL) at (-1.5,0) {$\SL$};

\node(RR) at (1.5,0) {$\QR$};

\node(QA) at (0,2) {$P_A$};
\draw[->] (QA) to node[inner sep = 8, left] {$\sL\circ\pL$} (SL);
\draw[->] (QA) to node[inner sep = 8, right] {$\qR\circ\pR$} (RR);

\end{tikzpicture}
\end{center}

\begin{definition}
Suppose that $\C$ has $\F$-pullbacks in $\C^\prime$.  The functor $\F$ is \emph{span tight} if for any $\F$-pullbacks $S$ and $Q$ of the same cospan, the unique span isomorphism $\Phi$ from $\F(S)$ to $\F(Q)$ is $\F(\Psi)$ for some span isomorphism $\Psi$ from $S$ to $Q$.
\end{definition}

\subsection{$\F$-Pullbacks of $\SurjSub$}

Suppose that $X$, $Y$, and $Z$ are smooth manifolds.  Suppose further that $f$ is a surjective submersion from $X$ to $Z$ and that $g$ is a surjective submersion from $Y$ to $Z$.  Proposition~\ref{prop:TransverseansFiberedProductDim} implies that the fibered product $X\times_ZY$ is an embedded submanifold of $X\times Y$.  Again, denote by $\rho_X$ and $\rho_Y$  the respective projections from $X\times Y$ to $X$ and $Y$, and let $\pi_X$ and $\pi_Y$ be their respective restrictions to $X\times_ZY$.  

\begin{proposition}\label{prop:TransverseansFiberedProductPB}
For any cospan $\(f,g\)$ in $\SurjSub$, the span $\(\pi_X, \pi_Y\)$ is a pullback in $\Dif$\! of $\(f,g\)$.  
 \end{proposition}

\begin{proof}
\begin{pf}
For any span $Q$ in $\Dif$ that is paired with the cospan $\(f,g\)$, take $\Phi$ to be the map from $\QA$ to $X\times Y$ that is given for any $a$ in $\QA$ by \[\Phi(a) = (\qL(a),\qR(a)).\]  This map is smooth as a product of smooth maps and unique since $\Dif$ has products.  Furthermore, for any $a$ in $\QA$, \[(f\circ\rho_X\circ\Phi)\!\(a\) = f\!\(\qL\!\(a\)\) \quad {\rm and} \quad (g\circ\rho_Y\circ\Phi)\!\(a\) = g\!\(\qR\!\(a\)\).\]  Since $Q$ is paired with $\(f,g\)$, $f\!\(\qL\!\(a\)\)$ is equal to $g\!\(\qR\!\(a\)\)$, and so $\Phi\!\(a\)$ is in $X\times_ZY$.  Therefore, the span $\(\pi_X, \pi_Y\)$ is a pullback in $\Dif$.
\end{pf}
\end{proof}

Note that while $\SurjSub$ is a subcategory of $\Dif$, the category $\SurjSub$ does not have pullbacks. Take $\F$ to be the inclusion functor from $\SurjSub$ to $\Dif$.  For any cospan $\(f, g\)$ in $\SurjSub$, where $f$ and $g$ have respective sources $X$ and $Y$ and both maps have target $Z$, Proposition~\ref{prop:TransverseansFiberedProductPB} implies that the span $\(\pi_X, \pi_Y\)$ is an $\F$-pullback of the cospan $\(f, g\)$.  Take $Q$ to be another $\F$-pullback of $(f,g)$.  There is a span isomorphism $\Phi$ from $\F(Q)$ to $\(\F(\pi_X), \F(\pi_Y)\)$, since both spans are pullbacks of $(\F(f),\F(g))$ in $\Dif$.  Every diffeomorphism is a surjective submersion, and this implies Theorem~\ref{SurjSubTight}. 

\begin{theorem}\label{SurjSubTight}
The inclusion functor from $\SurjSub$ to $\Dif$ is span tight.
\end{theorem}

\section{Generalized Span Category}\label{sec:genSpan}

Define the objects in ${\rm Span}\!\(\C,\F\)$ to be the objects in $\C$ and the isomorphism classes of spans in $\C$ to be the morphisms in ${\rm Span}\!\(\C,\F\)$.  For any isomorphism class of spans $[S]$ in ${\rm Span}\!\(\C, \F\)$, identify $\SR$ and $\SL$ respectively to be the source and target of $[S]$.  Define composition of isomorphism classes of spans by \[\left[S^{\1}\right] \circ \left[S^{\2}\right] = \left[S^{\1} \circ_P S^{\2}\right],\] where $S^{\1} \circ_P S^{\2}$ is an $\F$-pullback composite of $S^{\1}$ and $S^{\2}$.  Theorem~\ref{thm:SpanCatIfTight} is the main result of the section and the lemmata that follow simplify the proof of the theorem.

\begin{theorem}\label{thm:SpanCatIfTight}
If $\F$ is a span tight functor from $\C$ to $\C^\prime$, then ${\rm Span}\!\(\C, \F\)$ is a category.
\end{theorem}

If the functor $\F$ from $\C$ to $\C^\prime$ is span tight and $S$ and $Q$ are spans in $\C$ with $\SR$ equal to $\QL$, then there is an $\F$-pullback $P$ of the cospan $\(\sR, \qL\)$, and so there is an $\F$-pullback composite of $S$ and $Q$ along $P$.  The $\F$-pullback $P$ is, however, only defined up to a span isomorphism $\Phi$.  The following lemma shows that changing $P$, up to an isomorphism, changes the resulting composite span only up to a span isomorphism in $\C$.

\begin{lemma}\label{lemma:WDP}
Suppose that $\F$ is span tight, that $S$ and $Q$ are spans in $\C$, and that $S\circ_{P^{i}}Q$ is an $\F$-pullback composite, with $i$ equal to $1$ or $2$.  There is a span isomorphism $\Phi$ in $\C$ from $S\circ_{P^{\1}}Q$ to $S\circ_{P^{\2}}Q$. 
\end{lemma}

\begin{proof}
\begin{pf}
Since $P^{\1}$ and $P^{\2}$ are both $\F$-pullbacks of the cospan $(\sR, \qL)$, there is a span isomorphism $\Phi$ in $\C^\prime$ from $\F\!\(P^{\1}\)$ to $\F\!\(P^{\2}\)$.  Since $\F$ is span tight, there is a span isomorphism $\Psi$ in $\C$ from $P^{\1}$ to $P^{\2}$ with $\F(\Psi)$ equal to $\Phi$, and so \begin{equation}\label{V:Lem:diffFPBs}\pL^{\1} = \pL^{\2}\circ\Psi \quad {\rm and}\quad \pR^{\1} = \pR^{\2}\circ\Psi.\end{equation}  Left compose $\sL$ with both sides of the left equality of \eqref{V:Lem:diffFPBs} and $\qR$ with both sides of the right equality of \eqref{V:Lem:diffFPBs} to obtain the respective equalities \[\sL\circ\pL^{\1} = \sL\circ\pL^{\2}\circ\Psi \quad {\rm and}\quad \qR\circ\pR^{\1} = \qR\circ\pR^{\2}\circ\Psi\] that demonstrate that $\Psi$ is a span isomorphism from $S\circ_{P^{\1}}Q$ to $S\circ_{P^{\2}}Q$.
\end{pf}
\end{proof}

\begin{lemma}\label{lemma:WD}
Suppose that $\F$ is span tight, that $S^{i}$ and $Q^{i}$ are spans in $\C$, and that $S^{i}\circ_{P^{i}} Q^{i}$ is an $\F$-pullback composite, with $i$ equal to $1$ or $2$.  Suppose that $S^{\1}$ and $Q^{\1}$ are respectively span isomorphic to $S^{\2}$ and $Q^{\2}$.  There is a span isomorphism in $\C$ between spans $S^{\1}\circ_{P^{\1}} Q^{\1}$ and $S^{\2}\circ_{P^{\2}} Q^{\2}$.
\end{lemma}

This diagram includes many of the mappings involved in the proof of Lemma~\ref{lemma:WD}:

\medskip

\begin{center} 
\begin{tikzpicture}[->=stealth',node distance=2.25cm, auto, scale=.8]
\node(SA) at (1.5,0) {$\SA^{\2}$};
\node(SL) at (0,2) {$\SL^{\1}$};
\draw[->] (SA) to node[yshift = -6, inner sep = 2, left] {$\sL^{\2}$} (SL);

\node(QR) at (4.5,0) {$\QA^{\2}$};
\node(QA) at (3,2) {$\SR^{\1} = \QL^{\1}$};
\draw[->] (SA) to node[yshift = -6, inner sep = 2, right] {$\sR^{\2}$} (QA);
\draw[->] (QR) to node[yshift = -6, inner sep = 2, left] {$\qL^{\2}$} (QA);

\node(TA) at (6,2) {$\QR^{\1}$};
\draw[->] (QR) to node[inner sep = 8, right] {$\qR^{\2}$} (TA);

\node[](PAa) at (1.5,4) {$\SA^{\1}$};
\draw[->] (PAa) to node[yshift = 6, inner sep = 2, left] {$\sL^{\1}$} (SL);
\draw[->] (PAa) to node[yshift = 6, inner sep = 2, right] {$\sR^{\1}$} (QA);
\node[](PAb) at (4.5,4) {$\QA^{\1}$};
\draw[->] (PAb) to node[yshift = 6, inner sep = 2, left] {$\qL^{\1}$} (QA);
\draw[->] (PAb) to node[yshift = 6, inner sep = 2, right] {$\qR^{\1}$} (TA);
\node[](PA) at (3,6) {$\PA^{\1}$};
\draw[->] (PA) to node[yshift = 6, inner sep = 2, left] {$\pL^{\1}$} (PAa);
\draw[->] (PA) to node[yshift = 6, inner sep = 2, right] {$\pR^{\1}$} (PAb);

\node[](PAaa) at (3,-2) {$\PA^{\2}$};
\draw[->] (PAaa) to node[yshift = -6, inner sep = 2, left] {$\pL^{\2}$} (SA);
\draw[->] (PAaa) to node[yshift = -6, inner sep = 2, right] {$\pR^{\2}$} (QR);

\draw (PAa) to node[inner sep = 4, left] {$\alpha$}  (SA);
\draw (PAb)  to node[inner sep = 4, right] {$\beta$}  (QR);

\draw[white, line width=4pt] (PA) edge[out=0,in=0, looseness = 1.5] (PAaa);
\draw[] (PA) edge[out=0,in=0, looseness = 1.5, dashed] (PAaa);

\end{tikzpicture}
\end{center} 

\begin{proof}
\begin{pf}

Take $\alpha$ and $\beta$ to be span isomorphisms from $S^{\1}$ to $S^{\2}$ and from $Q^{\1}$ to $Q^{\2}$, respectively.  The span $P^{\1}$ is an $\F$-pullback of $\(\sR^{\1}, \qL^{\1}\)$.  Since $\alpha$ and $\beta$ are span morphisms, the span $\(\alpha\circ\pL^{\1},\beta\circ\pR^{\1}\)$ is paired with $\(\sR^{\2}, \qL^{\2}\)$.  Since $\F\!\(P^{\2}\)$ is a pullback of $\(\F\!\(\sR^{\2}\), \F\!\(\qL^{\2}\)\)$, there is a span morphism $\Phi_1$ in $\C^\prime$ from $\(\F\!\(\alpha\circ\pL^{\1}\),\F\!\(\beta\circ\pR^{\1}\)\)$ to $\F\!\(P^{\2}\)$.

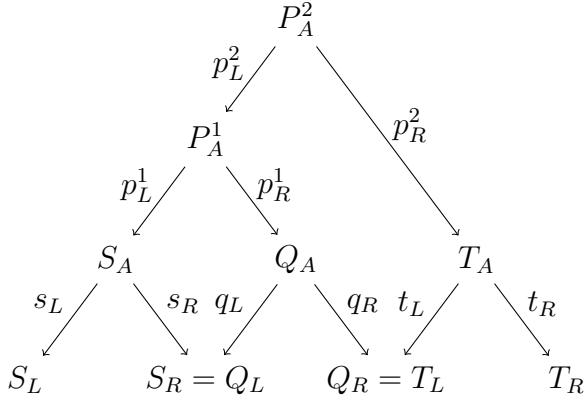
If $T$ is a span in $\C^\prime$ that is paired with $\(\F\!\(\sR^{\1}\), \F\!\(\qL^{\1}\)\)$, then there is a span morphism $\Phi_2$ in $\C^\prime$ from $T$ to $\F\!\(P^{\1}\)$.   The composite $\Phi_1\circ\Phi_2$ is a span morphism from $T$ to $\(\F\!\(\alpha^{-1}\circ\pL^{\2}\),\F\!\(\beta^{-1}\circ\pR^{\2}\)\)$, which is also paired with $\(\F\!\(\sR^{\1}\), \F\!\(\qL^{\1}\)\)$.  Uniqueness of the pullback of $\(\F\!\(\sR^{\1}\), \F\!\(\qL^{\1}\)\)$ up to a span isomorphism implies that there is a span isomorphism $\Phi_3$ in $\C^\prime$ from $\(\F\!\(\alpha^{-1}\circ\pL^{\2}\),\F\!\(\beta^{-1}\circ\pR^{\2}\)\)$ to $\F\!\(P^{\1}\)$.  Since $\F$ is span tight, there is a span isomorphism $\Psi$ in $\C$ such that $\F\!\(\Psi\)$ is $\Phi_3$.  Use the fact that $\Psi$ is a span isomorphism to obtain the equalities \begin{equation}\label{V:Lem:SpanIsoDifDif}\alpha^{-1}\circ \pL^{\2} = \pL^{\1}\circ \Psi \quad {\rm and}\quad \beta^{-1}\circ \pR^{\2} = \pR^{\1}\circ \Psi.\end{equation} The equalities \[\sL^{\2} = \sL^{\1}\circ \alpha^{-1} \quad {\rm and}\quad \qR^{\2} = \qR^{\1}\circ \beta^{-1}\] together with \eqref{V:Lem:SpanIsoDifDif} imply that \[\sL^{\2}\circ \pL^{\2} = \sL^{\1}\circ \alpha^{-1}\circ \pL^{\2} = \sL^{\1}\circ \pL^{\1}\circ \Psi\] and similarly that \[\qR^{\2}\circ \pR^{\2}  = \qR^{\1}\circ \pR^{\1}\circ \Psi,\] and so $\Psi$ is a span isomorphism from $S^{\2}\circ_{P^{\2}} Q^{\2}$ to $S^{\1}\circ_{P^{\1}} Q^{\1}$.
\end{pf}
\end{proof}

Lemma~\ref{lemma:WD} generalizes Lemma~\ref{lemma:WDP} and reduces to Lemma~\ref{lemma:WDP} when $S^1$ is equal to $S^2$, when $Q^1$ is equal to $Q^2$, and when $P^1$ and $P^2$ are pullbacks that are not necessarily equal to each other. 

\begin{lemma}\label{lemma:associative}
Suppose that $\F$ is span tight and that $S$, $Q$, and $T$ are spans in $\C$ with $\SR$ equal to $\QL$ and $\QR$ equal to $\TL$.  Suppose that $S\circ_{P^{\1}} Q$ and $Q\circ_{P^{\4}} T$ are $\F$-pullback composites and that $\(S\circ_{P^{\1}} Q\)\circ_{P^{\2}}T$ and $S\circ_{P^{\3}}\(Q\circ_{P^{\4}}T\)$ are also $\F$-pullback composites.  There is a span isomorphism from $\(S\circ_{P^{\1}} Q\)\circ_{P^{\2}}T$ to $S\circ_{P^{\4}}\(Q\circ_{P^{\3}}T\)$. 
\end{lemma}

The mappings involved in the proof of Lemma~\ref{lemma:associative} appear in this commutative diagram:

\medskip

\begin{center} 
\begin{tikzpicture}[->=stealth',node distance=2.25cm, auto, scale=.8]
\node(SL) at (-1.5,0) {$\SL$};
\node(SR) at (1.5,0) {$\SR = \QL$};
\node(SA) at (0,2) {$\SA$};
\draw[->] (SA) to node[yshift = 6, inner sep = 2, left] {$\sL$} (SL);
\draw[->] (SA) to node[yshift = 6, inner sep = 2, right] {$\sR$} (SR);

\node(QR) at (4.5,0) {$\QR = \TL$};
\node(QA) at (3,2) {$\QA$};
\draw[->] (QA) to node[yshift = 6, inner sep = 2, left] {$\qL$} (SR);
\draw[->] (QA) to node[yshift = 6, inner sep = 2, right] {$\qR$} (QR);

\node(TR) at (7.5,0) {$\TR$};
\node(TA) at (6,2) {$\TA$};
\draw[->] (TA) to node[yshift = 6, inner sep = 2, left] {$\tL$} (QR);
\draw[->] (TA) to node[yshift = 6, inner sep = 2, right] {$\tR$} (TR);

\node(PAa) at (1.5,4) {$\PA^{\1}$};
\draw[->] (PAa) to node[yshift = 6, inner sep = 2, left] {$\pL^{\1}$} (SA);
\draw[->] (PAa) to node[yshift = 6, inner sep = 2, right] {$\pR^{\1}$} (QA);

\node(PAb) at (3,6) {$\PA^{\2}$};
\draw[->] (PAb) to node[yshift = 6, inner sep = 2, left] {$\pL^{\2}$} (PAa);
\draw[->] (PAb) to node[yshift = 6, inner sep = 2, right] {$\pR^{\2}$} (TA);

\end{tikzpicture}
\end{center}

\noindent and in this commutative diagram:

\medskip

\begin{center} %
\begin{tikzpicture}[->=stealth',node distance=2.25cm, auto, scale = .8]
\node(SL) at (-1.5,0) {$\SL$};
\node(SR) at (1.5,0) {$\SR = \QL$};
\node(SA) at (0,2) {$\SA$};
\draw[->] (SA) to node[yshift = 6, inner sep = 2, left] {$\sL$} (SL);
\draw[->] (SA) to node[yshift = 6, inner sep = 2, right] {$\sR$} (SR);

\node(QR) at (4.5,0) {$\QR = \TL$};
\node(QA) at (3,2) {$\QA$};
\draw[->] (QA) to node[yshift = 6, inner sep = 2, left] {$\qL$} (SR);
\draw[->] (QA) to node[yshift = 6, inner sep = 2,  right] {$\qR$} (QR);

\node(TR) at (7.5,0) {$\TR$};
\node(TA) at (6,2) {$\TA$};
\draw[->] (TA) to node[yshift = 6, inner sep = 2,  left] {$\tL$} (QR);
\draw[->] (TA) to node[yshift = 6, inner sep = 2, right] {$\tR$} (TR);

\node(PAa) at (1.5,4) {$\PA^{\1}$};
\draw[->] (PAa) to node[yshift = 6, inner sep = 2, left] {$\pL^{\1}$} (SA);
\draw[->] (PAa) to node[yshift = 8, inner sep = 1, right] {$\pR^{\1}$} (QA);

\node(PAb) at (4.5,4) {$\PA^{\3}$};
\draw[->] (PAb) to node[yshift = 8, inner sep = 1,  left] {$\pL^{\3}$} (QA);
\draw[->] (PAb) to node[yshift = 6, inner sep = 2, right] {$\pR^{\3}$} (TA);

\node(PA) at (3,6) {$\PA$};
\draw[->] (PA) to node[yshift = 6, inner sep = 2, left] {$\pL$} (PAa);
\draw[->] (PA) to node[yshift = 6, inner sep = 2, right] {$\pR$} (PAb);

\end{tikzpicture}
\end{center}

\begin{proof}
\begin{pf}
Suppose that $P^{\1}$ is an $\F$-pullback of the cospan $\(\sR, \qL\)$, that $P^{\3}$ is an $\F$-pullback of the cospan $\(\qR, \tL\)$, and that $P$ is an $\F$-pullback of the cospan $\big(\pR^{\1}, \pL^{\3}\big)$, where \[\PL = \PA^{\1} \quad {\rm and}\quad \PR = \PA^{\3}.\]  Suppose further that $P^{\2}$ is an $\F$-pullback of the cospan $\big(\qR\circ \pR^{\1}, \tL\big)$ and that $P^{\4}$ is an $\F$-pullback of the cospan $\big(\sR, \qL\circ \pL^{\3}\big)$.  Since $P^{\2}$ is an $\F$-pullback of the cospan $\big(\qR\circ \pR^{\1}, \tL\big)$, the span $\big(\pR^{\1} \circ \pL^{\2}, \pR^{\2}\big)$ is paired with the cospan $\big(\qR, \tL\big)$, and so $\big(\F(\pR^{\1} \circ \pL^{\2}), \F(\pR^{\2})\big)$ is paired with the cospan $\big(\F(\qR), \F(\tL)\big)$.  The span $P^{\3}$ is an $\F$-pullback, which implies the existence of a span morphism $\Phi_1$ in $\C^\prime$ from $\big(\F(\pR^{\1} \circ \pL^{\2}), \F(\pR^{\2})\big)$ to $\F\big(P^{\3}\big)$.  The span $\big(\F\big(p^{\2}_L\big), \Phi_1\big)$ is paired with $\big(\F(p^{\1}_R), \F(p^{\3}_L)\big)$, and so there is a span morphism $\Phi_2$ in $\C^\prime$ from $\big(\F(p^{\2}_L), \Phi_1\big)$ to $\F(P)$.  

If $U$ is a span paired with $\big(q_R\circ p^{\1}_R, t_L\big)$, then there is a span morphism $\Phi_3$ in $\C^\prime$ from $\F(U)$ to $\F(P^{\2})$.  The composite $\Phi_2\circ \Phi_3$ is a span morphism in $\C^\prime$ from $\F(U)$ to $\big(\F(p_L), \F(p^{\3}_R\circ p_R)\big)$, and so $\big(\F(p_L), \F(p^{\3}_R\circ p_R)\big)$ is a pullback in $\C^\prime$ of the cospan $\big(\F(\qR\circ p_R^{\1}), \F(\tL)\big)$.  There is, therefore, a span isomorphism in $\C^\prime$ from $\big(\F(\pL), \F(\pR^{\3}\circ\pR)\big)$ to $\F\big(P^{\2}\big)$.   Span tightness of $\F$ implies that there is a span isomorphism $\Psi_1$ in $\C$ from $\big(\pL, \pR^{\3}\circ\pR\big)$ to $P^{\2}$, with \[\pR^{\2}\circ \Psi_1 = \pR^{\3}\circ \pR \quad\text{and so} \quad \tR\circ\pR^{\2}\circ \Psi_1 = \tR\circ\pR^{\3}\circ \pR.\] The equality \[\pL^{\2}\circ \Psi_1 = \pL \quad\text{implies that} \quad \sL\circ\pL^{\1}\circ\pL^{\2}\circ \Psi_1 = \sL\circ\pL^{\1}\circ\pL.\] The isomorphism $\Psi_1$ in $\C$ is, therefore, a span isomorphism, with \begin{equation*}\big(\sL\circ\pL^{\1}\circ\pL, \tR\circ\pR^{\3}\circ \pR\big) = \big(\sL\circ\pL^{\1}\circ \pL^{\2}\circ \Psi_1, \tR\circ\pR^{\2}\circ \Psi_1\big),\end{equation*} and so $\big(\sL\circ\pL^{\1}\circ\pL, \tR\circ\pR^{\3}\circ \pR\big)$ and $\(S\circ_{P^{\1}} Q\)\circ_{P^{\2}}T$ are span isomorphic.

A similar argument shows that $\big(\sL\circ\pL^{\1}\circ\pL, \tR\circ\pR^{\3}\circ \pR\big)$ and $S\circ_{P^{\4}} \(Q \circ_{P^{\3}}T\)$ are span isomorphic, where $P^{\4}$ is an $\F$-pullback of the cospan $\(\sR, \qL\circ\pL^{\3}\)$.  Since $\(S\circ_{P^{\1}} Q\)\circ_{P^{\2}}T$ and $S\circ_{P^{\4}} \(Q \circ_{P^{\3}}T\)$ are span isomorphic to the same span, Proposition~\ref{prop:spammap} implies Lemma~\ref{lemma:associative}.
\end{pf}
\end{proof}

\begin{proof}[Proof of Theorem~\ref{thm:SpanCatIfTight}]
\begin{pf}
To prove the theorem, it suffices to show that the composition of morphisms in ${\rm Span}\!\(\C, \F\)$ is well defined, satisfies the left and right unit laws, and is associative.

For any two isomorphism classes of spans $[S^{\1}]$ and $[S^{\2}]$, where the source of $\left[S^{\1}\right]$ is the target of $\left[S^{\2}\right]$, if $S^{\1}$ and $S^{\2}$ are, respectively, representatives of $\left[S^{\1}\right]$ and $\left[S^{\2}\right]$, then there is an $\F$-pullback $P$ of $\big(\sR^{\1}, \sL^{\2}\big)$.  Hence, there exists a composite $S^{\1} \circ_PS^{\2}$.  Lemma~\ref{lemma:WDP} implies that the equivalence class $\left[S^{\1} \circ_PS^{\2}\right]$ is independent of $P$.  Lemma~\ref{lemma:WD} additionally implies that $\left[S^{\1} \circ_PS^{\2}\right]$ is independent of the choice of representatives $S^{\1}$ and $S^{\2}$. Furthermore, $\SR^{\2}$ is the source of $\left[S^{\1}\right] \circ \left[S^{\2}\right]$ and $\SL^{\1}$ is the target of $\left[S^{\1}\right] \circ \left[S^{\2}\right]$, and so the composition $\circ$ is well defined. 

Suppose that $[S]$ is an isomorphism class of spans in $\C$ and that $\left[{\rm I}_{\SR}\right]$ is the isomorphism class of spans containing $\({\rm Id}_{\SR}, {\rm Id}_{\SR}\)$, where \[{\rm Id}_{\SR}\colon \SR \to \SR\] is the identity map from $\SR$ to $\SR$. Take $P$ to be the span $\({\rm Id}_{\SA}, \sR\)$.  For any span $Q$ in $\C^\prime$ that is paired with $\(\F\!\(\sR\), \F\!\({\rm Id}_{\SR}\)\)$, \[\F\!\(\sR\) \circ \qL = \F\!\({\rm Id}_{\SR}\) \circ \qR = \qR\] and so $\qL$ is a span morphism in $\C^\prime$ from $Q$ to $\F(P)$.  Given any other span morphism $\Phi$ in $\C^\prime$ from $Q$ to $\F(P)$, \[\qL = \F\!\({\rm Id}_{\SA}\)\circ \Phi= {\rm Id}_{\F(\SA)}\circ \Phi = \Phi.\] The span morphism in $\C^\prime$ from $Q$ to $\F(P)$ is unique, which implies that $\F\!\(P\)$ is a pullback in $\C^\prime$ of $\(\F\!\(\sR\), \F\!\({\rm Id}_{\SR}\)\)$, and so $P$ is an $\F$-pullback of $\(\sR, {\rm Id}_{\SR}\)$.  

The composite $S\circ_P {\rm I}_{\SR}$ is determined by these mappings:

\medskip

\begin{center} 
\begin{tikzpicture}[->=stealth',node distance=2.25cm, auto,  scale=.8]
\node[](SL) at (-3.6,0) {$\SL$};
\node[](SR) at (0,0) {$\SR$};
\node[](SA) at (-1.8,2) {$\SA$};
\draw[->] (SA) to node[yshift = 6, inner sep = 2, left] {$\sL$} (SL);
\draw[->] (SA) to node[yshift = 8, inner sep = 0, right] {$\sR$} (SR);

\node(RR) at (3.6,0) {$\SR$};
\node[](RA) at (1.8,2) {$\SR$};
\draw[->] (RA) to node[yshift = 8, inner sep = 0, left] {${\rm Id}_{\SR}$} (SR);
\draw[->] (RA) to node[yshift = 6, inner sep = 2, right] {${\rm Id}_{\SR}$} (RR);

\node(QA) at (0,4) {$\SA$};
\draw[->] (QA) to node[yshift = 6, inner sep = 2, left] {${\rm Id}_{\SA}$} (SA);
\draw[->] (QA) to node[yshift = 6, inner sep = 2, right] {$\sR$} (RA);

\end{tikzpicture}
\end{center}

\medskip

\noindent and is this diagram:

\begin{center} 
\begin{tikzpicture}[->=stealth',node distance=2.25cm, auto, scale=.8]
\node(SL) at (-1.8,0) {$\SL$};

\node(RR) at (1.8,0) {$\SR$};

\node(QA) at (0,2) {$\SA$};
\draw[->] (QA) to node[inner sep = 8, left] {$\sL$} (SL);
\draw[->] (QA) to node[inner sep = 8, right] {$\sR$} (RR);

\end{tikzpicture}
\end{center}

\noindent Since composition is well defined and $S\circ_P {\rm I}_{\SR}$ is span isomorphic in $\C$ to $S$, the composite $\left[S\right]\circ\left[{\rm I}_{\SR}\right]$ is equal to $\left[S\right]$.  Similar arguments will show that $\left[{\rm I}_{\SL}\right]\circ\left[S\right]$ is equal to $\left[S\right]$, and so ${\rm Span}\!\(\C, \F\)$ has both a right and left unit law.

Lemma~\ref{lemma:associative} implies that $\circ$ is associative.
\end{pf}
\end{proof}

\section{Examples}\label{sec:app}


\subsection{Categories that have Pullbacks} Suppose that $\C$ is a category that has pullbacks, and let $\F$ be the identity functor from $\C$ to $\C$.  The functor $\F$ is span tight and so ${\rm Span}(\C,F)$ is a category.  Since every $\F$-pullback of a cospan is a pullback of a cospan, the category ${\rm Span}(\C,\F)$ is the category ${\rm Span}(\C)$.  In this way, the concept of a generalized span category ${\rm Span}(\C, \F)$ generalizes the notion of a span category and reduces to it when $\C$ has pullbacks and $\F$ is the identity functor. 

\subsection{Smooth Manifolds and Surjective Submersions} Suppose that $\F$ is the inclusion functor from $\SurjSub$ to $\Dif$.  Theorems~\ref{SurjSubTight} and \ref{thm:SpanCatIfTight} together imply that ${\rm Span}(\SurjSub, \F)$ is a category.

\subsection{Classical Mechanics} In \cite{BWY}, we worked in the categories $\Riem$, whose objects are Riemannian manifolds and whose morphisms are surjective Riemannian submersions, and $\SympSurj$, whose objects are symplectic manifolds and whose morphisms are surjective Poisson maps.  Unlike $\SurjSub$, these categories are not subcategories of $\Dif$.  However, the forgetful functors from these categories into $\Dif$ are still span tight, and so it is possible to construct generalized span categories in these settings which are critical to the study of classical mechanics.  In \cite{BWY}, we introduced the notion of an \emph{augmented generalized span category}.  Such categories are critical to the categorification of classical mechanics and the study of the functoriality of the Legendre transformation.


\end{document}